\documentclass[12pt]{article}

\usepackage{epsfig}
\usepackage{amsfonts}
\usepackage{amssymb}
\usepackage{amstext}
\usepackage{amsmath}
\usepackage{xspace}
\usepackage{theorem}
\usepackage{graphicx}

\makeatletter
\newenvironment{proof}{{\bf Proof:  }}{\hfill\rule{2mm}{2mm}}

\newcommand{\junk}[1]{}

\newtheorem{theorem}{Theorem}
\newtheorem{lemma}[theorem]{Lemma}
\newtheorem{conjecture}[theorem]{Conjecture}

\newtheorem{claim}[theorem]{Claim}
\newtheorem{corollary}[theorem]{Corollary}
\newtheorem{definition}[theorem]{Definition}

\newcommand{\Comb}{\ensuremath{\textsc{comb}}}
\newcommand{\supp}{\ensuremath{\text{supp}}}

\title{On the Adjacency Spectra of Hypertrees}

\author{Gregory J. Clark and Joshua N. Cooper\\
\small Department of Mathematics\\[-0.8ex]
\small University of South Carolina\\
}
\begin{document}
\maketitle

\begin{abstract}
We extend the results of \cite{Zha} to show that $\lambda$ is an eigenvalue of a $k$-uniform hypertree $(k \geq 3)$ if and only if it is a root of a particular matching polynomial for a connected induced subtree.  We then use this to provide a spectral characterization for power hypertrees.  Notably, the situation is quite different from that of ordinary trees, i.e., $2$-uniform trees.  We conclude by presenting an example (an $11$ vertex, $3$-uniform non-power hypertree) illustrating these phenomena.
\end{abstract}

{\bf Keywords}: Hypergraph; Characteristic Polynomial; Matching Polynomial; Power Graph.

\section{Introduction}

The following beautiful result was shown in \cite{Zha}: the set of roots of a certain matching polynomial of a $k$-uniform hypertree (an acyclic $k$-uniform hypergraph) is a subset of its homogeneous adjacency spectrum.  

\begin{theorem} 
\label{t:main} (\cite{Zha})
$\lambda$ is a nonzero eigenvalue of a hypertree $H$ with the corresponding eigenvector ${\bf x}$ having all elements nonzero if and only if it is a root of the polynomial \[\varphi(H) = \sum_{i=0}^m (-1)^i |{\cal M}_i| x^{(m-i)r}\] where ${\cal M}_i$ is the collection of all $t$-matchings of $H$. 
\end{theorem}

In the present work, we show how to obtain {\em all} of the eigenvalues of a hypertree, and use this description to give a spectral characterization of ``power'' hypertrees (defined below).  Here the notion of a hypergraph's eigenpairs is the homogeneous adjacency spectrum, \`{a} la Qi \cite{Qi}, Lim \cite{Lim}, and Cooper-Dutle \cite{Coo}.  We extend Theorem \ref{t:main} as follows to describe the spectrum of a hypertree, answering the main open question in \cite{Zha}.

\begin{theorem}
\label{t: main2}
Let ${\cal H}$ be a $k$-uniform hypertree, for $k \geq 3$; $\lambda$ is a nonzero eigenvalue of  ${\cal H}$ if and only if there exists an induced subtree $H \subseteq {\cal H}$ such that $\lambda$ is a root of the polynomial $\varphi(H)$.  
\end{theorem}

We note that Theorem \ref{t: main2} is not true for (2-uniform) trees (c.f.~Cauchy's interlacing theorem), making it an unusual example of a result in spectral hypergraph theory that fails in the graph case.  The necessity for $k \geq 3$ is established by Theorem \ref{t: tree} (\cite{Zha}, \cite{Zho}). 

Below, we reserve the use of the term ``hypergraph'' for the case of $k$-uniform hypergraphs with $k\geq 3$, and use the language of ``graphs'' exclusively for $k=2$.  We also make use of the nomenclature $k$-graph and $k$-tree to mean $k$-uniform hypergraph and $k$-uniform hypertree, respectively. We maintain much of the notation of \cite{Zha} and refer the interested reader to their paper.  In the next section, we provide necessary definitions and prove Theorem \ref{t: main2}.  We then use this result to show that power trees are characterized by their spectra being cyclotomic (Theorem \ref{t: main3}) and provide an example demonstrating these phenomena.
\section{Proof of Theorem \ref{t: main2}}

A vector is \emph{totally nonzero} if each coordinate is nonzero and an eigenpair $(\lambda, {\bf x})$ is totally nonzero if $\lambda \neq 0$ and ${\bf x}$ is a totally nonzero vector. Given a vector ${\bf x} \in \mathbb{C}^n$ the \emph{support} $\supp({\bf x})$ is the set of all indices of non-zero coordinates of ${\bf x}$.  Let ${\bf x}^\circ$ denote the totally nonzero projection (by restriction) of ${\bf x}$ onto $\mathbb{C}^{|\supp({\bf x})|}$. For ease of notation, we assume that the coordinate indices of vectors agree with the vertex labeling of the hypergraph under consideration.  We denote the induced subgraph of $\cal H$ on $U \subseteq V(\mathcal{H})$ by \[{\cal H}[U] = (U, \{v_1\dots v_k \in E(\mathcal{H}) : v_i \in U\})\] and write $H \sqsubseteq \mathcal{H}$ to mean $H = \mathcal{H}[U]$ for some $U \subseteq V(\mathcal{H})$.

The following establishes the forward direction of Theorem \ref{t: main2}.

\begin{lemma}
\label{l: supp}
Let $(\lambda, {\bf x})$ be a nonzero eigenpair of the normalized adjacency matrix of a $k$-uniform hypergraph ${\cal H}$.  Then $(\lambda, {\bf x}^\circ)$ is a totally nonzero eigenpair of ${\cal H}[\supp({\bf x})]$.
\end{lemma}

\begin{proof}
As $(\lambda, {\bf x})$ is an eigenpair of ${\cal H}$,
\[
\sum_{i_2. i_2, \dots, i_k=1}^n a_{ji_2i_3\dots i_k}x_{i_2}x_{i_3}\dots x_{i_k} = \lambda x_j^{k-1}
\] 
for $j \in [n]$ by definition.  Let $m = |\supp({\bf x})|$ and suppose without loss of generality that $\supp({\bf x}) = [m]$.  For $j \in [m]$ we have
\begin{align*}
\lambda (x^\circ)_j^{k-1} &= \lambda x_j^{k-1} \\
&= \sum_{i_2. i_2, \dots, i_k=1}^n a_{ji_2i_3\dots i_k}x_{i_2}x_{i_3}\dots x_{i_k} \\
&= \sum_{i_2. i_2, \dots, i_k=1}^m a_{ji_2i_3\dots i_k}x_{i_2}x_{i_3}\dots x_{i_k}. \\
\end{align*}
Thus, $(\lambda, {\bf x}^\circ)$ is an eigenpair of ${\cal H}[m]$ by definition; moreover, $(\lambda, {\bf x}^\circ)$ is totally nonzero, as each coordinate of ${\bf x}^\circ$ is nonzero by construction. 
\end{proof}

The following result, from \cite{Coo}, shows how the eigenvalues of a disconnected hypergraph arise from the eigenvalues of its components.

\begin{theorem}(\cite{Coo})
\label{t: subgraph}
Let $H$ be a $k$-graph that is the disjoint union of hypergraphs $H_1$ and $H_2$.  Then as sets, $spec(H) = spec(H_1) \cup spec(H_2)$.
\end{theorem}

The \emph{matching polynomial} $\varphi(H)$ of $H$ is defined by
\[
\varphi(H) = \sum_{i=0}^m (-1)^i |{\cal M}_i| x^{(m-i)r}
\]
where ${\cal M}_i$ is the collection of all $t$-matchings of $H$. We now show that $\varphi$ is multiplicative over connected components. 

\begin{claim}
\label{c: varphi}
If ${\cal H} = \bigsqcup_{i=1}^t H_i$ is a disjoint union of k-uniform hypertrees then $\varphi({\cal H}) = \prod_{i=1}^t \varphi(H_i)$.
\end{claim}

\begin{proof}  Clearly, the result follows inductively if it is true for $t = 2$. Denote matching numbers $m(H_1) = m_1$ and $m(H_2) = m_2$.  Indexing the sum of $H_1$ and $H_2$ by $i$ and $j$, respectively, we compute 
\begin{align*}
\varphi(H_1)\varphi(H_2) &= \left(\sum_{i = 0}^{m_1} (-1)^i |{\cal M}_i| x^{(m_1 - i)k}\right) \left(\sum_{j = 0}^{m_2} (-1)^i |{\cal M}_j| x^{(m_2 - i)k}\right) \\
&= \sum_{\substack{0 \leq i \leq m_1 \\ 0\leq j \leq m_2}}(-1)^{i+j} |{\cal M}_i||{\cal M}_j| x^{(m_1 + m_2 - (i+j))k}.
\end{align*}
Let \[\varphi(H_1 \sqcup H_2) := \sum_{\ell = 0}^m (-1)^\ell |{\cal M}_\ell| x^{(m - \ell)k}.\]  Since matchings of $H$ are unions of matchings of $H_1$ and $H_2$, $m = m_1 + m_2$ is the size of the the largest matching of $H$.  Furthermore, for any $0 \leq \ell \leq m$,
\[
|{\cal M}_{\ell}| = \sum_{\substack{0 \leq i \leq m_1 \\ 0 \leq j \leq m_2 \\ i + j = \ell}} |{\cal M}_i||{\cal M}_j|,
\]
since a matching of $\ell$ edges in $H$ consists of a matching of $i$ edges in $H_1$ and $j$ edges in $H_2$, where $i + j = \ell$.   Substituting yields  \[\varphi(H_1)\varphi(H_2)=\sum_{\ell = 0}^m (-1)^\ell |{\cal M}_\ell| x^{(m - \ell)k} = \varphi(H_1 \sqcup H_2)\] as desired.
\end{proof}

Recall from \cite{Zha} that a {\em pendant} edge is a $k$-uniform edge with exactly $k-1$ vertices of degree 1. 

\begin{claim}
\label{c: tree}
Let ${\cal H}$ be an $r$-uniform hypertree.  If $H \subseteq {\cal H}$ is a sub-hypertree then there exists a sequence of edges $(e_1, e_2, \dots, e_t)$ such that $e_i$ is a pendant edge of $H_i$ where $H_0 := {\cal H}$, $H_i := H_{i-1}-e_i$, and $H_t := H$.
\end{claim}

\begin{proof}
We prove our claim by induction on $| |E({\cal H})| - |E(H)||$.  The base case of $| |E({\cal H})| - |E(H)|| = 0$ is immediate.  Suppose that $||E({\cal H})| - |E(H)|| = 1$.  Since $H$ is a tree, $H$ is necessarily connected.  In particular, $H$ is formed by removing a pendant edge of ${\cal H}$, and the claim follows.
\end{proof}

In keeping with the notation of $\cite{Zha}$, let $s({\cal H})$ be the set of all sub-hypertrees of $\cal{H}$. 

\begin{theorem} (\cite{Zha})
\label{t: tree}
For any $k$-uniform hypertree ${\cal H}$ where $k \geq 3$, the roots of \[\prod_{H_i \in s(H)} \varphi(H_i)\] are eigenvalues of ${\cal H}$.  Moreover, the largest root is the spectral radius of $H$.
\end{theorem}
 
We now present a proof of Theorem \ref{t: main2}.  

\begin{proof}
Theorem \ref{t:main} establishes the case when $\lambda$ is nonzero and has an eigenvector which is totally nonzero.

Let $(\lambda, {\bf x})$ be a nonzero eigenpair of ${\cal H}$.  From Claim \ref{l: supp}, it follows that $(\lambda, {\bf x}^\circ)$ is a totally nonzero eigenpair of ${\cal H}[\supp({\bf x})]$.  For the moment, let \[H = {\cal H}[\supp({\bf x})] \subseteq {\cal H}.\]  If $H$ is connected we apply Theorem \ref{t:main} to conclude that $\lambda$ is a root of $\varphi(H)$ as desired.  If instead $H$ is disconnected, then we may write $H = \bigsqcup_{i=1}^t H_i$, where each $H_i$ is connected.  Appealing to Theorem \ref{t: subgraph}, $\lambda$ is a root of $\varphi(H)$, whence $(x-\lambda) \mid \varphi(H_i)$ for some $i$.  Indeed, $(\lambda, ({\bf x^\circ})^\circ)$ is a totally nonzero eigenpair of $H_i$.  Therefore, $H_i$ is a connected subgraph of $\cal H$ we apply Theorem \ref{t:main} to conclude that $\lambda$ is a root of $\varphi(H_i)$, as desired.

Now suppose that $\lambda$ is a root of $\varphi(H)$ for $H \subseteq {\cal H}$.   The case of $H = {\cal H}$ was established by Theorem \ref{t:main}: suppose further that $H \subsetneq {\cal H}$.  If $H$ is connected then by Theorem \ref{t:main} $(\lambda, {\bf x^\circ})$ is a totally nonzero eigenpair of $H$.  By Claim \ref{c: tree} and Theorem \ref{t: tree} we conclude that $\lambda$ is an eigenvalue of ${\cal H}$.   Suppose even further that $H = \bigsqcup_{i=1}^t H_i$ where $H_i$ is a connected component.  We have shown in Claim \ref{c: varphi}, \[\varphi(H) = \prod_{i=1}^t \varphi(H_i).\]  As $\lambda$ is a root of $\varphi(H)$, it is true that $\lambda$ is a root of $\varphi(H_i)$ for some $i$.  Appealing to Theorem \ref{t:main}, once more we have that $(\lambda, {\bf x}^\circ)$ is a totally nonzero eigenpair of $H_i$; therefore, $(\lambda, {\bf x})$ is an eigenpair of ${\cal H}$ by Claim \ref{c: tree} and Theorem \ref{t: tree} as desired. 
\end{proof}

\section{The spectra of power trees} 

The following generalizes the definition of powers of a hypergraph from \cite{Hu}.

\begin{definition}
Let $H$ be an $r$-graph for $r \geq 2$.  For any $k \geq r$, the $k^\textrm{th}$ power of $G$, denoted $H^k$, is a $k$-uniform hypergraph with edge set \[E(H^k) = \{e \cup \{v_{e,1}, \dots, v_{e, k-r}\} : e \in E(G)\},\] and vertex set \[V(H^k) = V(G) = V(G) \cup \{i_{e,j} : e \in E(G), j \in [k-r]\}.\]
\end{definition}

In other words, one adds exactly enough new vertices (each of degree $1$) to each edge of $H$ so that $H^k$ is $k$-uniform.  Note that, if $k=r$, then $H^k = H$.  Adhering to this nomenclature we refer to a power of a $2$-tree simply as a {\em power tree}.  In this section we prove the following characterization of power trees.

\begin{theorem}
\label{t: main3}
Let ${\cal H}$ be a $k$-tree.  Then $\sigma({\cal H}) \subseteq \mathbb{R}[\zeta_k]$ if and only if ${\cal H}$ is a power tree, where $\zeta_k$ is a principal $k^\textrm{th}$ root of unity.
\end{theorem}

We recall the following Theorem from Cooper-Dutle.

\begin{theorem} \cite{Coo} The (multiset) spectrum of a $k$-cylinder is invariant under multiplication by any $k^\textrm{th}$ root of unity.
\end{theorem}

One can show by straightforward induction that a $k$-tree is a $k$-cylinder, so its spectrum is symmetric in the above sense.  The following result, from \cite{Zho}, shows that power trees have spectra which satisfy a much more stringent condition: they are cyclotomic, in the sense that they belong to $\mathbb{R}[\zeta_k]$.

\begin{theorem}
\label{t: Zho}
\cite{Zho} If $\lambda \neq 0$ is an eigenvalue of any subgraph of $G$, then $\lambda^{2/k}$ is an eigenvalue of $G^k$ for $k \geq 4$.
\end{theorem}

We restate Theorem \ref{t: Zho} with the additional  assumption that the underlying graph is a tree; the proof is easily obtained by applying Claim \ref{c: tree} to the proof of Theorem \ref{t: Zho} appearing in \cite{Zho}.

\begin{corollary}
\label{c: Zho}
If $\lambda \neq 0$ is an eigenvalue of any subgraph of a tree $T$, then $\lambda^{2/k}$ is an eigenvalue of $T^k$ for $k \geq 3$.
\end{corollary}

Note that Theorem \ref{t: main3} provides a converse to Corollary \ref{c: Zho} in the case of power trees.  In particular, appealing to Theorem \ref{t: main2}, Corollary \ref{c: Zho}, and the fact that the spectrum of a graph is real-valued, we have that the spectrum of a power tree is a subset of $\mathbb{R}[\zeta_k]$.  All that remains to be shown is that if a $k$-tree is not a power tree then it has a root in $\mathbb{C} \setminus \mathbb{R}[\zeta_k]$.  To that end, we introduce the $k$-comb.

Let $\Comb_k$ be the $k$-graph where \[\Comb_k = ([k^2], \{[k] \cup\{\{i + tk : 0 \leq t \leq k-1\} : i \in [k]\}\}\}.\]

We refer to $\Comb_k$ as the $k$-comb.  By the definition of power tree, a non-power tree $H$ must contain an edge $e$ incident to a family $\mathcal{F}$ consisting of at least three other edges which are mutually disjoint.  This edge $e$, together with $\mathcal{F}$, form a connected induced subgraph $H^\prime$ of $H$ which is the $k^\textrm{th}$ power of a $t$-comb for $t = |\mathcal{F}| \geq 3$.  It is straightforward to see that $\varphi(H^\prime)(x) = \varphi(\Comb_t^k)(x) = \varphi(\Comb_t)(x^{k/t})$, since matchings in $H^\prime$ are simply $k^\textrm{th}$ powers of matchings in $\Comb_t$; therefore, roots of $\varphi(H^\prime)$ are $k^\textrm{th}$ roots of reals if and only if the roots of $\varphi(\Comb_t)$ are $t^\textrm{th}$ roots of reals.  We presently show that the spectrum of the $k$-comb is not contained within the $k^\textrm{th}$ cyclotomic extension of $\mathbb{R}$, completing the proof of Theorem \ref{t: main3}.

\begin{lemma}
There exists a root $\lambda$ of $\varphi(\Comb_k)$ for $k \geq 3$ such that $\lambda \in \mathbb{C} \setminus \mathbb{R}[\zeta_i]$.
\end{lemma}

\begin{proof}
Let ${\cal H}$ be a $k$-comb where $k \geq 3$.  By a simple counting argument, 
\[
\varphi({\cal H}) = \left(\sum_{i=0}^k (-1)^i \binom{k}{i} \alpha^{k-i}\right) - \alpha^{k-1}
\] 
where $\alpha = x^k$.  Appealing to the binomial theorem we have \[\varphi({\cal H}) = (1-\alpha)^k - \alpha^{k-1}.\]  Let $\beta = \alpha^{-1}$.  Setting $\varphi({\cal H}) = 0$ yields 
\begin{equation} \label{eq1}
(\beta -1)^k = \beta.
\end{equation}
It is easy to see that (\ref{eq1}) has precisely one solution when $k\geq 3$ is odd and precisely two solutions when it is even.  In either case, as the number of solutions is strictly less than $k$ it follows that there must be a non-real solution and the claim follows. 
\end{proof}

\section{Concluding Remarks}

We conclude our note by presenting an example demonstrating Theorems \ref{t: main2} and \ref{t: main3}.

Consider the 3-uniform hypergraphs
\begin{align*}
{\cal H}_1 &=([9], \{\{1,2,3\},\{1,4,7\},\{2,5,8\},\{3,6,9\}\}) = \Comb_3 \\
{\cal H}_2 & = ([9], \{\{1,2,3\},\{1,4,7\},\{3,6,9\},\{1,10,11\}\})\\
{\cal H}_3&= ([11], E({\cal H}_1) \cup E({\cal H}_2)).
\end{align*}

We have computed 
\begin{align*}
\phi({\cal H}_1) &= x^{567}(x^9-4x^6+3x^3-1)^{81}(x^6-3x^3+1)^{81}(x^3-2)^{27}(x^3-1)^{147} \\
\phi({\cal H}_2) &=x^{999}(x^6-4x^3+2)^{81}(x^6-3x^3+1)^{54}(x^3-3)^{27}(x^3-2)^{63}(x^3-1)^{75} \\
\phi({\cal H}_3) &=x^{3767}(x^9-5x^6+5x^3-2)^{243}(x^9-4x^6+3x^3-1)^{162}(x^6-4x^3+2)^{162}\\
&\qquad \cdot (x^6-3x^3+1)^{135}(x^3-3)^{27}(x^3-2)^{180}(x^3-1)^{483}.
\end{align*}  

Let $P_n$ and $S_n$ denote the 3-uniform loose path and star with $n$ edges, respectively.  We list the non-trivial induced subgraphs of ${\cal H}_3$ and their matching polynomials below.

\begin{center}
\label{t: subgraphs}
\begin{tabular}{|c|c|}
\hline
 $H \sqsubseteq {\cal H}_3$ & $\varphi(H)$ \\
 \hline 
 $P_1 = S_1$ & $x^3-1$ \\ \hline
  $P_2 = S_2$ & $x^3-2$ \\ \hline
$P_3$  & $x^6-3x^3+1$ \\ \hline
$S_3$ & $x^3-3$ \\ \hline
${\cal H}_1$ & $x^9-4x^6+3x^3-1$\\ \hline
 ${\cal H}_2$ & $x^6-4x^3+2$ \\ \hline
 ${\cal H}_3$ & $x^9-5x^6+5x^3-2$\\ \hline
\end{tabular}
\end{center}

Figure \ref{fig1} gives a drawing of ${\cal H}_3$ (the striped subgraph is ${\cal H}_1$) and a plot of the roots of $\phi({\cal H}_3)$, with a circle centered at each root in the complex plane whose area is proportional to the multiplicity of the root.  Notice that, despite the rotational symmetry (turning by a third), the cubes of the roots are not all real, i.e., some of the roots do not lie on the the rays with argument $0$, $2\pi/3$, or $4 \pi/3$.

\begin{figure}[h]
\centering
\begin{minipage}{.5\textwidth}
 \centering
\includegraphics[width=.8\linewidth]{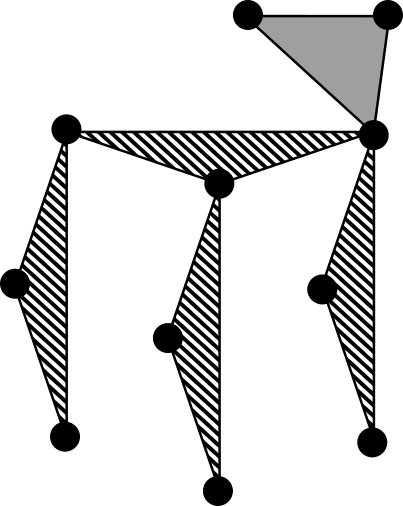}
\end{minipage}%
\begin{minipage}{.5\textwidth}
 \centering
 \includegraphics[width=.8\linewidth]{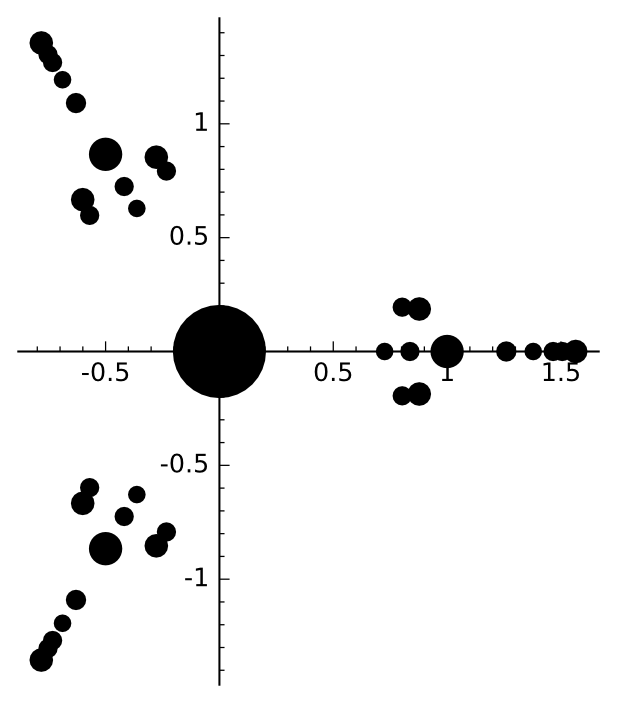}
\end{minipage}
\caption{\label{fig1} ${\cal H}_3$ and its spectrum.}
\end{figure}

Observe that each matching polynomial divides wholly into the characteristic polynomials.  A priori, this is a symptom of the matching polynomials having distinct roots.  A preliminary question: what can one say about the roots of $\varphi(H_1)$ and $\varphi(H_2)$ for $k$-trees $H_1 \sqsubseteq H_2$?  With this question in mind we conjecture the following.

\begin{conjecture}
If  $H \sqsubseteq {\cal H}$ are $k$-trees for $k \geq 3$ then $\varphi(H) \mid \phi({\cal H})$.  In particular, if $H \sqsubseteq {\cal H}$ then $\phi(H) \mid \phi({\cal H})$.
\end{conjecture}

\end{document}